\documentclass[12pt,final]{amsart}

\usepackage{txfonts}
\usepackage{amssymb,dsfont}
\usepackage{enumerate}
\usepackage{slashed}
\usepackage{fancyhdr}
\usepackage{aliascnt}
\usepackage{color}
\usepackage{hyperref}
\usepackage[square, numbers, sort&compress]{natbib}
\usepackage[dvips,dvipsnames]{xcolor}
\usepackage[dvips]{graphicx}
\usepackage[all,cmtip]{xy}
\usepackage{pdfsync}

\makeatletter
\newcommand{\autorefcheckize}[1]{%
  \expandafter\let\csname @@\string#1\endcsname#1%
  \expandafter\DeclareRobustCommand\csname relax\string#1\endcsname[1]{%
    \csname @@\string#1\endcsname{##1}\wrtusdrf{##1}}%
  \expandafter\let\expandafter#1\csname relax\string#1\endcsname
}
\makeatother

\hypersetup{
linkcolor=red,citecolor=green,filecolor=magenta,urlcolor=cyan}

\newcommand{\abs}[1]{\left\lvert#1\right\rvert}
\newcommand{\set}[1]{\left\{#1\right\}}
\newcommand{\hin}[2]{\left\langle#1,#2\right\rangle}
\newcommand{\rin}[2]{\left(#1,#2\right)}
\newcommand{\field}[1]{\mathbb{#1}}
\newcommand{\kh}[1]{\left(#1\right)}
\newcommand{\zkh}[1]{\left[#1\right]}
\newcommand{\R}{\field{R}}
\newcommand{\Com}{\field{C}}
\newcommand{\Lg}[1]{\mathrm{#1}}
\newcommand{\vect}[1]{\mathbf{#1}}
\newcommand{\To}{\longrightarrow}
\newcommand{\Rmn}[1]{\uppercase\expandafter{\romannueral#1}}

\theoremstyle{plain}
\newtheorem{theorem}{Theorem}[section]

\newtheorem{known}{Theorem}

\newaliascnt{lem}{theorem}
\newtheorem{lem}[lem]{Lemma}
 \aliascntresetthe{lem}

\newaliascnt{cor}{theorem}
\newtheorem{cor}[cor]{Corollary}
 \aliascntresetthe{cor}

\newaliascnt{prop}{theorem}

 \aliascntresetthe{prop}

\newaliascnt{exam}{theorem}
\newtheorem{exam}[exam]{Example}
 \aliascntresetthe{exam}

\newtheorem*{conj}{Conjecture}

\theoremstyle{remark}
\newtheorem{rem}{Remark}[section]

\theoremstyle{definition}

\numberwithin{equation}{section}

\allowdisplaybreaks

\begin{document}

\title[sphere theorems]{Some sharp differential sphere theorems for nonnegative scalar curvature manifolds
}


\author{Qing Cui}
\address{School of Mathematics\\
Southwest Jiaotong University\\
611756 Chengdu, Sichuan, China}
\email{cuiqing@swjtu.edu.cn}

\author[Sun]{Linlin Sun}
\address{School of Mathematics and Statistics \& Computational Science Hubei Key Laboratory, Wuhan University, 430072 Wuhan, Hubei, China}
\email{sunll@whu.edu.cn}

\thanks{This work is partially supported by National Natural Science Foundation
 of China (Grant No. 11601442) and
Fundamental Research Funds for the Central Universities (Grant No. 2682016CX114, 2042018kf0044
).}


\date{\today}

\begin{abstract}
In this paper, we obtain several new intrinsic and extrinsic differential sphere theorems via Ricci flow.
 For intrinsic case, we show that a closed simply connected $n(\ge 4)$-dimensional Riemannian
 manifold $M$ is diffeomorphic to $\Lg{S}^n$
 if one of the following
 conditions holds pointwisely:
$$
 (i)\   R_0>\kh{1-\frac{24(\sqrt{10}-3)}{n(n-1)}}K_{max};\quad
 \ (ii)\   \frac{Ric^{[4]}}{4(n-1)}>\kh{1-\frac{6(\sqrt{10}-3)}{n-1}}K_{max}.$$
 Here $K_{max}$, $Ric^{[k]}$ and $R_0$ stand for the maximal sectional curvature,
 the $k$-th weak Ricci curvature and
 the normalized scalar curvature.
 For extrinsic case, i.e.,  when $M$ is a closed simply connected $n(\ge 4)$-dimensional submanifold immersed in $\bar{M}$.
 We prove that $M$ is diffeomorphic to $\Lg{S}^n$ if it satisfies some pinching curvature conditions. The only involved extrinsic quantities in our pinching conditions are the maximal sectional curvature $\bar K_{max}$ and the squared norm of mean curvature vector $\abs{H}^2$. More precisely,
 we show that $M$ is  diffeomorphic to $\Lg{S}^n$
 if one of the following conditions
 holds:
\begin{itemize}
  \item[(1)] $R_0\ge \kh{1-\frac{2}{n(n-1)}}\bar{K}_{max}   +\frac{n(n-2)}{(n-1)^2}\abs{H}^2$, and strict
 inequality is achieved at some point;
 \item[(2)] $\dfrac{Ric^{[2]}}{2}\ge (n-2)\bar K_{max}+\frac{n^2}{8}\abs{H}^2,$ and strict
 inequality is achieved at some point;
 \item[(3)] $\dfrac{Ric^{[2]}}{2} \ge\frac{n(n-3)}{n-2}\kh{\bar K_{max}+\abs{H}^2},$ and strict
 inequality is achieved at some point.
 \end{itemize}
 It is worth pointing out that, in the proof of extrinsic case, we apply suitable complex orthonormal frame
  and simplify the calculations considerably.
 We also emphasize that both of the pinching constants in  (2) and (3) are optimal for $n=4$.

 \vspace{2ex}


\end{abstract}%

\maketitle

\section{Introduction}
It is a basic problem in Riemannian geometry  to classify closed Riamannian manifolds in the category
of either topology, diffeomorphism, or isometry under some curvature conditions. Among a huge literature on this problem,
the uniqueness of sphere  under pinched curvatures
accounts for a large proportion. One of the reasons for studying uniqueness of sphere is the simpleness of its topology.
These uniqueness results are usually called topology sphere theorems (in the homeomorphism sense), differential sphere theorems
(in the diffeomorphism sense),
and isometry (or rigidity) sphere theorems (in the isometry sense).

Suppose $M$ is a closed  $n$-dimensional Riemannian manifold. If $n=2$ and $M$ has positive Gaussian curvature, then
 one can easily see from Gauss-Bonnet formula that $M$ must be a topological sphere. Since the differential structure is unique on
 a 2-sphere, $M$ must be diffeomorphic to a standard 2-sphere $\Lg{S}^2$. When $n=3$, the Riemannian curvature tensor
 is uniquely determined by the Ricci tensor. Hamilton \cite{Ham82} showed that if a closed 3-dimensional manifold has a metric with positive Ricci curvature,
 then it must be diffeomorphic to a spherical space form. Moreover, if $M$ is simply connected, $M$ must be diffeomorphic to $\Lg{S}^3$. Hamilton \cite{Ham86} classified all closed 3-dimensional Riemannian manifold with nonnegative Ricci curvature.
 Therefore, in this paper, we focus our attention on the dimension $n\ge 4$ and study sphere theorems with pinched curvatures.

The study of sphere theorems under pinched sectional curvatures goes back to a question of Hopf.
In 1951, Rauch \cite{Rau51} showed that a closed simply connected  Riemannian manifold with globally $\delta$-pinched $(\delta \approx 0.75)$
sectional curvature  is homeomorphic to a sphere. Rauch also proposed a question of what the optimal pinching constant should be. Berger \cite{Ber60}
and Killingenberg \cite{Kli61} proved that $\delta = \frac{1}{4}$ is the optimal pinching constant. Since on a sphere of arbitrary dimension, the differential
structure is not necessarily unique, it is natural to ask that if $\frac{1}{4}$-pinched sectional curvature is
necessary for a differential sphere? This question was finally answered by Brendle and Schoen \cite{BreSch09} via the Ricci flow.

Another important differential sphere theorem via Ricci flow is due to B\"ohm and Wilking \cite{BohWil08}. They proved that closed manifolds with $2$-positive curvature operator are spherical space forms.  Moreover, Berger \cite{Ber60} classified  all manifolds with weakly $1/4$-pinched curvatures in the homeomorphism sense.   Brendle and Schoen \cite{BreSch08} provided a classification, up to a diffeomorphism,  of all manifolds with weakly $1/4$-pinched curvatures.
For more sphere theorems under pinched
sectional curvatures, we refer the reader to a good survey book of Brendle \cite{Bre10a}.

It is well known that the  complex projective space $\mathbb CP^n$ with Fubini-Study metric has exactly pointwise $\frac{1}{4}$-pinched
sectional curvature (see also \autoref{exam}). Therefore, Brendle-Schoen's theorem is optimal for even dimension.
It is natural to study sphere theorems under other pinched curvature conditions. In 1990's, Yau collected
some open problems and  he wrote in Problem 12 (\cite{Yau93}):

\emph{`` The famous pinching problem says that on a compact simply connected manifold if
$K_{min} > \frac{1}{4}K_{max}$, then the manifold is homeomorphic to a sphere.
If we replace
$K_{max}$ by normalized scalar curvature, can we deduce similar pinching results? "}

Classical examples (see \cite[Example 1]{GuXu12}, see also \autoref{exam} in this paper) show
that the pinching constant is at least $\frac{n-1}{n+2}$. Therefore
Yau's problem can be written in a more concrete way (\cite[Yau Conjecture 1]{GuXu12}):
\begin{conj}[Yau 1990]
  Let $(M^n, g)$ be a closed simply connected Riemannian manifold.
  Denote by $R_0$ the normalized scalar curvature of $M^n$. If
$K_{min} > \frac{n-1}{n+2}R_0$, then $M^n$ is diffeomorphic to a standard sphere $\Lg{S}^n$.
\end{conj}

If $K_{min}>\kh{1-\frac{6}{n^2-n+6}}R_0,\  n\geq4$, Gu and Xu \cite{GuXu12} proved $M$ must be diffeomorphic
to a standard sphere, which partially answered Yau's problem. Moreover, if $M$ is an Einstein manifold,
Gu and Xu \cite{XuGu14} proved the pinching constant $\frac{n-1}{n+2}$ is optimal and gave an isometric sphere theorem.
When the dimension $n=4$, Costa and Ribeiro Jr.
\cite{CosRib14} proved Yau's conjecture. They actually used a weaker assumption
by replacing sectional curvature by biorthogonal curvature condition. We can prove when
$K_{min}>\left(1-\frac{12}{n^2-n+12}\right)R_0$, $M$ must be diffeomorphic to $\Lg{S}^n$.
However, when we finish this paper, we know from Professor Hong-Wei Xu that he and his collaborators
 obtained the same result \cite{GuXuZha17} independently. We would like to thank Professor Hong-Wei Xu for sending their manuscript \cite{GuXuZha17}. For readers' convenience, we still give a complete proof of this result in Section 3 (see \autoref{Kmin}).

It is also interesting to study sphere theorems with normalized scalar curvature pinched by $K_{max}$.
Gu and Xu \cite[Theorem 1]{GuXu12} showed that if $R_0> \frac{12}{5n(n-1)}K_{max}$, $n\ge 4$, then $M$ is diffeomorphic
to a spherical space form. Based on an example of $\mathbb OP^2$, the authors also posed a Conjecture (see \cite[Conjecture 1]{GuXu12}):
\begin{conj}
  Let $M^n(n\ge 4)$ be a closed and simply connected Riemannian manifold.
If $R_0 > \frac{3}{5}K_{max}$, then $M$ is diffeomorphic to $\Lg{S}^n$.
\end{conj}
We also get  a new differential sphere theorem in this direction:

 \begin{theorem}\label{Kmax}
  Let $M^n (n\ge 4)$ be a closed and simply connected Riemannian manifold.  If
$$
R_0>\kh{1-\frac{24(\sqrt{10}-3)}{n(n-1)}}K_{max},
$$
then $M$ is diffeomorphic to $\Lg{S}^n$.
\end{theorem}
\begin{rem}
Under the assumtion
\begin{align*}
  R_0>\kh{1-\frac{6}{n(n-1)}}K_{max},
\end{align*}
we can prove $M$ has positive isotropic curvature, see \autoref{rem1}.
 Gu-Xu-Zhao \cite{GuXuZha17} also obtained this result independently.
\end{rem}

For pinched  Ricci curvature and sectional curvature, we also have the following sphere theorem.

\begin{theorem}\label{2thRic}
  Let $M^n (n\ge 4)$ be a closed and simply connected Riemannian manifold. If
$$
\frac{Ric_M^{[4]}}{4(n-1)}>\kh{1-\frac{6(\sqrt{10}-3)}{n-1}}K_{max},
$$
then $M$ is diffeomorphic to $\Lg{S}^n$.
\end{theorem}
\begin{rem}
Gu-Xu-Zhao \cite{GuXuZha17} actually proved $M$ is diffeomorphic to
$ \Lg{S}^n$ when $M$ satisfies $$
\frac{Ric_M}{n-1}>\kh{1-\frac{3}{2(n-1)}}K_{max}.$$
\end{rem}

\vspace{2ex}

It is also of interest to  study sphere theorems for  submanifolds. In recent years, many authors  investigated
 related problems and plenty of works were obtained (e.g. \cite{Hui84, Hui86, AndBak10, LiuXuYe13, XuGu10, XuGu14, XuGu13, GuXu12, XuTia11, Bar11} and therein).   We also get sphere theorems for submanifolds corresponding to
 \autoref{Kmax} and \autoref{2thRic}, see \autoref{subkmin}, \autoref{subkmax} and \autoref{subric}.
Besides these results, we use complex orthonormal frames  to obtain the following new sphere theorems. The assumptions
of these theorems only involve $R_0$, $Ric^{[2]}$, $\bar{K}_{max}$ and $\abs{H}^2$.

We prove the following three theorems which are
generalizations of
Gu-Xu's results \cite[Theorem 3, Theorem 4]{GuXu12},  Xu-Gu's result \cite[Theorem 1.1]{XuGu10}, Anderws-Baker's result \cite[Theorem 1]{AndBak10},  Liu-Xu-Ye-Zhao's result \cite[Corollary 1.2]{LiuXuYe13} and Xu-Tian's result \cite[Theorem 1.1]{XuTia11}.
\begin{theorem}\label{subR0}
Suppose $M^n (n\geq4)$ is a closed and simply connected  submanifold of $\bar{M}^N$ satisfying
\begin{align*}
R_0\ge \kh{1-\frac{2}{n(n-1)}}\bar{K}_{max}   +\frac{n(n-2)}{(n-1)^2}\abs{H}^2,
\end{align*}
with strict inequality at some point, then $M$ is diffeomorphic to $\Lg{S}^n$.
\end{theorem}

\begin{theorem}\label{subric1}
Suppose $M^n (n\geq4)$ is a closed and simply connected submanifold of $\bar{M}^N$ satisfying
\begin{align*}
\frac{Ric^{[2]}}{2}\ge (n-2)\bar K_{max}+\frac{n^2}{8}\abs{H}^2,
\end{align*}
with strict inequality  at some point, then $M$ is  diffeomorphic to $\Lg{S}^n$.

\end{theorem}

The pinching condition in \autoref{subric1} is optimal. In fact, when $\bar M$ is the space form $F^N(c), c>0$, Ejiri \cite{Eji79} obtained a rigidity theorem for minimal submanifolds under the pinching condition
\begin{align*}
Ric_M>(n-2)c.
\end{align*}
Xu-Gu \cite{XuGu13} obtained an extension of Ejiri's results for constant mean curvature submanifolds in the space form $F^N(c)$ under the condition
\begin{align*}
Ric_M>(n-2)\left(c+\abs{H}^2\right)>0.
\end{align*}
They also obtained a topological sphere theorem for general submanifolds in the space form $F^N(c), c\geq0$ under the same pinching condition mentioned above by using Lawson-Simons theory for stable integral currents \cite{LawSim73, Xin84}. Motivated by these facts, the authors posed the following Conjecture (c.f., \cite[Conjecture A]{XuGu13}):

\begin{conj}Let $M^n (n\geq4)$ be a closed and simply connected orientated submanifold in the space form $F^N(c)$. If $Ric_M>(n-2)\left(c+\abs{H}^2\right)>0$, then $M$ is diffeomorphic to $\Lg{S}^n$.
\end{conj}

Here is a generalization of Gu-Xu's result \cite[Theorem 4.2]{XuGu13}.

\begin{theorem}\label{subric2}
Suppose $M^n (n\geq4)$ is a closed and simply connected submanifold of $\bar{M}^N$ satisfying
\begin{align*}
\frac{Ric^{[2]}}{2} \ge\frac{n(n-3)}{n-2}\kh{\bar K_{max}+\abs{H}^2},
\end{align*}
with strict inequality  at some point, then $M$ is diffeomorphic to $\Lg{S}^n$.
\end{theorem}

\begin{rem}The Bonnet-Myers theorem \cite{Mye41} claimed that every complete Riemannian manifold with Ricci curvature bounded from below by a positive constant is compact. For complete noncompact Riemannian manifold with quasi-positive sectional curvature, the soul theorem \cite{CheGro72, GroMey69, Per94} claimed that such manifold is diffeomorphic to the Euclidean space. Thus, one can consider the sphere theorems for complete Riemannian manifolds with similar curvature pinching conditions in the above theorems.
\end{rem}

This paper is organized as follows. In Section 2, we list
 some notations and known facts. In Section 3, we prove some intrinsic differential sphere theorems with pinched normalized scalar curvatures
 and pinched Ricci curvatures. In Section 4, we study a  Riemannian manifold immersed into another and give several new extrinsic
 topology sphere theorems and differential sphere theorems.

\vspace{2ex}
\noindent
{\bf Acknowledgement:} We would like to thank Dr. Jun Sun for useful discussions and suggestions.

\vspace{2ex}

\section{Preliminaries}
In this section, we will fix some notations and list several known facts which will be used in next two sections.

Let $\rin{M^n} {\hin{\ }{\ }}$ be a Riemannian manifold, $\nabla$ be the Levi-Civita connection related to
$\hin{\ }{\ }$ and $R$ be the Riemannian curvature tensor defined by
\begin{align*}
R(X,Y)\coloneqq[\nabla_X,\nabla_Y]-\nabla_{[X,Y]},\quad\forall X,Y\in TM.
\end{align*}
Denote
\begin{align*}
R(X,Y,Z,W)\coloneqq \hin{R(X,Y)W}{Z}.
\end{align*}
Define
$$K(X,Y)\coloneqq R(X,Y,X,Y),\quad \forall X, Y\in TM.$$
Denote $K(X,Y)$ by $K(\pi)$ if $X,Y$ are orthonormal and  $\pi=span\set{X, Y}$.
By the linearity and symmetry of $R$, it is easy to check the following identities.
\begin{lem}\label{lem1}
  For all $X, Y, Z, W \in TM$ and $a,b \in \mathbb R$, we have
\begin{align}
K\kh{X+Y, X-Y} = &4K\kh{X, Y},\notag\\
K\kh{X, Y+Z} + K\kh{X, Y-Z} = &2\kh{K\kh{X, Y} +K\kh{X, Z}},\notag\\
K\kh{aX, bY} = &a^2b^2 K\kh{X, Y},\notag\\
\label{sc4}4R(X,Y,X,Z)=&K(X,Y+Z)-K(X,Y-Z),\\
\label{sc5}24R(X,Y,Z,W)=&K(X+Z,Y+W)+K(X-Z,Y-W)+K(Y+Z,X-W)\\
&+K(Y-Z,X+W)-K(X+Z,Y-W)-K(X-Z,Y+W)\notag\\
&-K(Y+Z,X+W)-K(Y-Z,X-W).\notag
\end{align}
\end{lem}
Identities \eqref{sc4} and \eqref{sc5} actually were first used by Karcher \cite{Kar70} to give a short proof
of Berger's curvature tensor estimate.

Let $\kh{\bar{M}^N, \bar{g}} (N\ge n)$  be another Riemannian manifold such that there exists an isometrically immersion
$$
f: \kh{M^n, \hin{\ }{\ }} \to \kh{\bar{M}^N, \bar{g}}.
$$
When we do calculation on the submanifold, we always omit $f$ and also write $\bar{g}$ as $\hin{\ }{\ }$.
Let $\set{e_1,\cdots, e_N}$ be a local orthonormal frame on $\bar{M}$ such that
$\set{e_1,\cdots,e_n}$ form a local orthonormal frame of $M$. Let $\set{\omega^1, \cdots, \omega^n}$
be the coframe of $\set{e_1, \cdots, e_n}$.
Define $\bar{R}$ and $\bar{K}$ on $\bar{M}$ similarly as those on $M$. In what follows, without special explanation,
$i, j, k, l$ will always range from $1$ to $n$ and $\alpha, \beta, \gamma$ will always range from
$n+1$ to $N$. The second fundamental form is defined to be
$$
B=h^\alpha_{ij} \omega^i\otimes \omega^j \otimes e_\alpha.
$$
The squared norm of $B$ is $\abs{B}^2 =\sum_{i,j,\alpha}\kh{h^\alpha_{ij}}^2$.
Write $H^\alpha = \frac{1}{n}\sum_i h_{ii}^\alpha$, the mean curvature vector
 is given by $\vect{H} =H^\alpha e_\alpha$, and the (normalized) mean curvature is $H = \sqrt{\sum_\alpha \kh{H^\alpha}^2}$.

The Gauss equation can be written as
\begin{align*}
R_{ijkl} = \bar{R}_{ijkl} +\sum_{\alpha}\kh{h_{ik}^\alpha h_{jl}^\alpha - h_{il}^\alpha h_{jk}^\alpha},
\end{align*}
where $R_{ijkl}=R(e_i,e_j,e_k,e_l)$ and $\bar{R}_{ijkl}=\bar{R}(e_i,e_j,e_k,e_l)$.
In tensor language, Gauss equation also can be written as
\begin{align}\label{Gausseq}
R = \bar{R}^T +\frac{1}{2}\sum_{\alpha}h^{\alpha}\circledwedge h^{\alpha}\coloneqq \bar{R}^T +\frac12B\circledwedge B,
\end{align}
where $\bar{R}^T$ means the  restriction of $\bar{R}$ on $TM$, $\circledwedge$ denotes the
Kulkarni-Nomizu product  of two symmetric (0,2)-tensor $a$ and $b$ which defined in local coordinates by
$$
\kh{a\circledwedge b}_{ijkl} \coloneqq a_{ik}b_{jl} -a_{il}b_{jk}-a_{jk}b_{il} +a_{jl}b_{ik}.
$$

Fix $p\in M$, $X,Y\in T_pM$,  the following notations will be used throughout this paper:
\begin{align*}
  K_{min}(p) = \min_{\pi\subset T_pM} K(\pi), \quad\quad K_{max}(p) = \max_{\pi\subset T_pM} K(\pi),\\
  Ric(X,Y) =\sum_{i} R(X,e_i,Y,e_i), \quad Ric_{jj} = Ric(e_j,e_j),\quad  R_0= \frac{\sum_{i,j}R_{ijij}}{n(n-1)},\\
  \zkh{e_{i_1}, \cdots, e_{i_k}} =span\set{e_{i_1}, \cdots, e_{i_k}}, \quad\quad \forall 1\le i_1<i_2<\cdots <i_k\le n,\\
  Ric^{[k]}\zkh{e_{i_1}, \cdots, e_{i_k}} = \sum_{j=1}^k Ric_{i_ji_j},\quad\quad
  Ric^{[k]}_{min}(p)= \min_{\zkh{e_{i_1}, \cdots, e_{i_k}} \subset T_pM} Ric^{[k]}\zkh{e_{i_1}, \cdots, e_{i_k}} (p),
\end{align*}
where $Ric^{[k]}\zkh{e_{i_1}, \cdots, e_{i_k}}$ is called $k${\it-th weak Ricci curvature} of $\zkh{e_{i_1}, \cdots, e_{i_k}}$
which was first introduced by Gu-Xu in \cite{GuXu12}. One can also give similar notations as above on $\bar{M}$.
Since all our calculations is local (at $p$), we will always omit the letter ``$p$" in what follows.

 Complexify $TM$ to $T^{\mathbb C}M$ and assume $\varepsilon_1, \cdots, \varepsilon_n$ is a local orthonormal frame
 of $T^{\mathbb C}M$. Extend $R, \bar{R}, B$ and $\hin{\ }{\ }$ $\mathbb C$-linearly and denote by
\begin{align*}
h^\alpha_{i\,\bar{j}} = \hin{B(\varepsilon_i, \bar{\varepsilon}_j)}{e_\alpha}, \quad
  R_{ij\,\bar{i}\bar{j}} = R(\varepsilon_i, \varepsilon_j, \bar{\varepsilon}_i, \bar{\varepsilon}_j),
  \quad Ric_{i\,\bar{i}} = \sum_{j=1}^n R_{ij\,\bar{i}\bar{j}}.
\end{align*}
It is easy to check
\begin{align*}
  h^\alpha_{i\,\bar{i}}\in \mathbb R, \quad h^\alpha_{i\,\bar{j}}=\overline{h^\alpha_{\bar{i}j}}, \quad R_{ij\,\bar{i}\bar{j}}\in \mathbb R,
  \quad \sum_{i,j=1}^n R_{ij\,\bar{i}\bar{j}} = n(n-1)R_0.
\end{align*}
A direct computation via  the complex linearity gives the following complex Gauss equation, for $i\neq j$,
\begin{align}\label{complexGauss}
R_{ij\,\bar i\bar j}=&\bar{R}_{ij\,\bar i\bar j}+\sum_\alpha\kh{h^{\alpha}_{i\,\bar i}h^{\alpha}_{j\,\bar j}-h^{\alpha}_{i\,\bar j}h^{\alpha}_{\bar ij}}\\
=&\bar{R}_{ij\,\bar i\bar j}+\abs{H}^2+\sum_\alpha\kh{H^{\alpha}\left(\mathring{h}^{\alpha}_{i\,\bar i}
+\mathring{h}^{\alpha}_{j\,\bar j}\right)+\mathring{h}^{\alpha}_{i\,\bar i}\mathring{h}^{\alpha}_{j\,\bar j}
-\abs{\mathring{h}^{\alpha}_{i\,\bar j}}^2},\notag
\end{align}
where $\mathring{h}^\alpha_{i\bar{j}}= h^\alpha_{i\bar{j}}-H^\alpha \delta_{i\bar j}$.
Therefore, the complex Ricci curvature is given by
\begin{align}\label{complexRicci}
Ric_{i\bar i}=\sum_{j=1}^n\bar{R}_{ij\,\bar{i}\bar{j}}+(n-1)\abs{H}^2+\sum_\alpha\kh{(n-2)H^{\alpha}\mathring{h}^{\alpha}_{i\,\bar{i}}
-\sum_{k=1}^n\abs{\mathring{h}^{\alpha}_{i\,\bar k}}^2}.
\end{align}

The curvature operator $\mathcal{R}:\Lambda^2TM\To\Lambda^2TM$ is defined as follows:
\begin{align*}
\hin{\mathcal{R}(X\wedge Y)}{Z\wedge W}\coloneqq R(X,Y,Z,W).
\end{align*}
A linear subspace $V\in T^{\Com}M$ is called {\it totally isotropic} if $g(v,v)=0,$ for all $v\in V$. In other words, for all $v= X+\sqrt{-1}Y\in V$,
\begin{align*}
\abs{X}^2-\abs{Y}^2=\hin{X}{Y}=0.
\end{align*}

 To each complex 2-plane $\sigma\in\Lambda^2T^{\Com}M$ the  complex sectional curvature $K(\sigma)$ is defined to be
\begin{align*}
K(\sigma)\coloneqq\dfrac{\hin{\mathcal{R}(z\wedge w)}{\bar z\wedge\bar w}}{\abs{z\wedge w}^2},
\end{align*}
where $\sigma=span_{\Com}\{z,w\}$. It is obvious that $K(\sigma)\in \R$.  $K(\sigma)$ is called {\it isotropic curvature} if $\sigma$ is totally isotropic.
The concept of isotropic curvature was first introduced by Micallef and Moore \cite{MicWan93}.

It is easy to check that, for every totally isotropic $2$-plane, there exists an orthonormal four-frame $\set{e_1, e_2, e_3, e_4}$, such that
\begin{align*}
\sigma=span_{\Com}\set{e_1+\sqrt{-1}e_2,e_3+\sqrt{-1}e_4}.
\end{align*}
Moreover, by $\mathbb C$-linearity of $\mathcal{R}$ and $\hin{\ }{\ }$, we have
\begin{align*}
&4K(\sigma)\\
=&\hin{\mathcal{R}\left((e_1+\sqrt{-1}e_2)\wedge(e_3+\sqrt{-1}e_4)\right)}{(e_1-\sqrt{-1}e_2)\wedge(e_3-\sqrt{-1}e_4)}\\
=&\hin{\mathcal{R}\left(e_1\wedge e_3-e_2\wedge e_4+\sqrt{-1}(e_1\wedge e_4+e_2\wedge e_3)\right)}{e_1\wedge e_3-e_2\wedge e_4-\sqrt{-1}(e_1\wedge e_4+e_2\wedge e_3)}\\
=&\hin{\mathcal{R}(e_1\wedge e_3-e_2\wedge e_4)}{e_1\wedge e_3-e_2\wedge e_4}+\hin{\mathcal{R}(e_1\wedge e_4+e_2\wedge e_3)}{e_1\wedge e_4+e_2\wedge e_3}\\
=&R_{1313}+R_{2424}-2R_{1324}+R_{1414}+R_{2323}+2R_{1423}\\
=&R_{1313}+R_{1414}+R_{2323}+R_{2424}-2R_{1234},
\end{align*}
where we have used Bianchi identity in the last equality.
  When $M$ has positive isotropic curvature, Micallef and Moore proved the following theorem.
\begin{known}[Micallef-Moore \cite{MicMoo88}]\label{toposphere}
  Let $M$ be a closed $n(\ge 4)$-dimensional
Riemannian manifold. Assume for every  orthonormal four-frame $\{e_1, e_2, e_3, e_4\}$, the following inequality holds
\begin{align*}
R_{1313}+R_{1414}+R_{2323}+R_{2424}-2R_{1234}>0.
\end{align*}
 Then
$\pi_k(M)=0$ for  $2 \le k \le \zkh{\frac{n}{2}}$.
In particular,
if $M$ is simply connected, then $M$ is homeomorphic to a sphere.
\end{known}

When $M\times\R$ has nonnegative isotropic curvature, i.e., (c.f. \cite{Bre08})
\begin{align}\label{diffineq}
R_{1313} + \lambda^2 R_{1414} + R_{2323} + \lambda^2 R_{2424} - 2\lambda R_{1234}>0
\end{align}
for all orthonormal four-frames $\set{e_1,e_2,e_3,e_4}$ and all $\lambda\in [-1, 1]$,   we have the following differential sphere theorem.
\begin{known}[Brendle \cite{Bre08}]\label{diffsphere2} Let $(M,g_0)$ be a
closed Riemannian manifold of dimension
$n\ge 4$ such that $M\times\R$ has positive isotropic curvature.
 Then the
normalized Ricci flow with initial metric $g_0$ exists for all time and converges
to a constant curvature metric as $t\to \infty$.
\end{known}
\begin{rem}
 \autoref{diffsphere2} is also true if one can verify inequality \eqref{diffineq} for $\lambda\in[0,1]$.
 Actually, if inequality \eqref{diffineq} holds for $\lambda\in[0,1]$, then for $\mu\in[-1,0]$,
consider orthonormal four-frame $\set{ e_1, e_2, e_3, -e_4}$, we have
\begin{align*}
&R_{1313} + \mu^2 R_{1414} + R_{2323} + \mu^2 R_{2424} - 2\mu R_{1234}\\
=&R_{1313} + \mu^2 R_{1414} + R_{2323} + \mu^2 R_{2424} -2(-\mu) R(e_1,e_2,e_3,-e_4)>0.
\end{align*}
\end{rem}

Seshadri \cite{Ses09} studied  the classification of closed Riemannian manifolds with nonnegative isotropic curvature. When $M\times\R^2$ has nonnegative isotropic curvature, i.e.,  (c.f. \cite{BreSch09})
\begin{align}\label{eq:classif}
R_{1313} + \lambda^2 R_{1414}
+ \mu^2 R_{2323} + \lambda^2\mu^2 R_{2424}-2\lambda\mu R_{1234}\ge 0,
\end{align}
for all points $p\in M$, all orthonormal four-frames $\set{e_1, e_2, e_3, e_4}\subset T_pM,$ and all $\lambda,\mu\in[-1,1]$, or equivalently $M$ has nonnegative complex sectional curvature (c.f. \cite[Remark 3.3]{MicWan93} or \cite[Proposition 17.8]{Bre10a}), we have the following classification theorem.

\begin{known}[Brendle-Schoen \cite{BreSch08}]\label{classif}
  Let $M$ be a closed, locally irreducible Riemannian manifold of dimension $n\ge 4$.
  If $M \times \mathbb R^2$ has nonnegative isotropic curvature, then one of the following
statements holds:
\begin{enumerate}[(i)]
\item $M$ is diffeomorphic to a spherical space form;\\
\item $n=2m$ and the universal cover of $M$ is a K\"{a}hler manifold biholomorphic to
$\Com P^m$;\\
\item the universal cover of $M$ is isometric to a compact symmetric space.
\end{enumerate}
\end{known}
\begin{rem}
Similar to the remark after \autoref{diffsphere2}, this classification theorem is true if we can verify the condition \eqref{eq:classif} for all four-frame $\set{e_1,e_2,e_3,e_4}$ and all $\lambda,\mu\in[0,1]$.
\end{rem}

\vspace{2ex}

\section{Sphere theorems for pinched curvatures}
In this section, we will prove  the intrinsic sphere theorems listed in the introduction.
Before we prove these theorems, we give a useful lemma.
\begin{lem}\label{lem2}
Let $\{e_1, e_2, e_3, e_4\}$ be any orthonormal four-frame, then we have
\begin{align*}
12R_{1234} = &4\sum_{1\le i<j \le 4} R_{ijij} -2\kh{R_{1313} +R_{1414} +R_{2323}+R_{2424}}\\
&-\kh{K(e_1+e_3, e_2-e_4) + K(e_1-e_3, e_2+e_4) + K(e_2+e_3, e_1+e_4) +K(e_2-e_3, e_1-e_4)}.
\end{align*}

\end{lem}
\begin{proof}First note that
$$
\set{\frac{e_1+e_3}{\sqrt{2}}, \frac{e_1-e_3}{\sqrt{2}}, \frac{e_2+e_4}{\sqrt{2}}, \frac{e_2-e_4}{\sqrt{2}}}, \quad
\set{\frac{e_1+e_4}{\sqrt{2}}, \frac{e_1-e_4}{\sqrt{2}}, \frac{e_2+e_3}{\sqrt{2}}, \frac{e_2-e_3}{\sqrt{2}}}
$$
are two orthonormal basises of ${\rm span}\set{e_1,e_2,e_3,e_4}$.
Therefore, by \autoref{lem1}, we have
\begin{align}\label{sum1}
4\sum_{1\le i<j\le 4} R_{ijij}=&
  K(e_1+e_3, e_1-e_3) + K(e_1+e_3, e_2+e_4) +K(e_1+e_3, e_2-e_4)\\
  &+K(e_1-e_3, e_2+e_4) +K(e_1-e_3, e_2-e_4) + K(e_2+e_4, e_2-e_4).\notag\\
\label{sum2}  4\sum_{1\le i<j\le 4} R_{ijij}=& K(e_1+e_4, e_1-e_4) + K(e_1+e_4, e_2+e_3) +K(e_1+e_4, e_2-e_3)\\
  &+K(e_1-e_4, e_2+e_3) +K(e_1-e_4, e_2-e_3) + K(e_2+e_3, e_2-e_3).\notag
\end{align}
 Set $X=e_1, Y=e_2, Z=e_3, W=e_4$ in \eqref{sc5}, we have
 \begin{align*}
 &24R_{1234}\\
 =&K(e_1+e_3,e_2+e_4)+ K(e_1-e_3,e_2-e_4)+K(e_2+e_3,e_1-e_4)
+K(e_2-e_3,e_1+e_4)\\
&-K(e_1+e_3,e_2-e_4)-K(e_1-e_3,e_2+e_4)-K(e_2+e_3,e_1+e_4)-K(e_2-e_3,e_1-e_4)\\
=&K(e_1+e_3, e_1-e_3) + K(e_1+e_3, e_2+e_4) +K(e_1+e_3, e_2-e_4)\\
  &+K(e_1-e_3, e_2+e_4) +K(e_1-e_3, e_2-e_4) + K(e_2+e_4, e_2-e_4)\\
  &+K(e_1+e_4, e_1-e_4) + K(e_1+e_4, e_2+e_3) +K(e_1+e_4, e_2-e_3)\\
  &+K(e_1-e_4, e_2+e_3) +K(e_1-e_4, e_2-e_3) + K(e_2+e_3, e_2-e_3)\\
&-2\kh{K(e_1+e_3,e_2-e_4)+K(e_1-e_3,e_2+e_4)+K(e_2+e_3,e_1+e_4)+K(e_2-e_3,e_1-e_4)}\\
&-K(e_1+e_3, e_1 -e_3)-K(e_2+e_4, e_2 -e_4)-K(e_1+e_4, e_1 -e_4)-K(e_2+e_3, e_2 -e_3)\\
= &8\sum_{1\le i<j \le 4} R_{ijij} -4\kh{R_{1313} +R_{1414} +R_{2323}+R_{2424}}\\
&-2\kh{K(e_1+e_3, e_2-e_4) + K(e_1-e_3, e_2+e_4) + K(e_2+e_3, e_1+e_4) +K(e_2-e_3, e_1-e_4)}.
 \end{align*}
In the last equality, we have used \eqref{sum1} and \eqref{sum2}.
\end{proof}

The following theorem obtained by Gu-Xu-Zhao \cite{GuXuZha17} independently. We list a proof here for
reader's convenience.
\begin{theorem}\label{Kmin}
Let $M^n (n\geq4)$ be a closed and simply connected Riemannian manifold. Assume  the following pinching condition holds,
$$
 K_{min}>\left(1-\dfrac{12}{n^2-n+12}\right)R_0,
$$
then $M$ is  diffeomorphic to $\Lg{S}^n$.
\end{theorem}
\begin{proof}[Proof of Theorem \autoref{Kmin}]
By \autoref{diffsphere2}, it is sufficient to prove \eqref{diffineq} holds for every orthonormal four-frame $\set{e_1, e_2, e_3, e_4}$
and $\lambda\in [0,1]$.
By \autoref{lem2}, we have
\begin{align*}
&12\left(R_{1313}+R_{2323}+R_{1234}\right)\\
=&4\sum_{1\le i<j \le 4} R_{ijij}-2\left(R_{1414}+R_{2424}\right)+10\left(R_{1313}+R_{2323}\right)  \\
&-\kh{K(e_1+e_3,e_2-e_4)+K(e_1-e_3,e_2+e_4)}\\
&-\kh{K(e_2+e_3,e_1+e_4)+K(e_2-e_3,e_1-e_4)}\\
=&4\sum_{1\le i<j \le 4} R_{ijij}-2\left(R_{1414}+R_{2424}\right)+10\left(R_{1313}+R_{2323}\right)  \\
&-\kh{4\sum_{1\le i<j\le 4}R_{ijij} -K(e_1+e_3,e_2+e_4)-K(e_1-e_3,e_2-e_4)-4R_{1313}-4R_{2424}}\\
&-\kh{4\sum_{1\le i<j\le 4}R_{ijij} -K(e_2+e_3,e_1-e_4)-K(e_2-e_3,e_1+e_4)-4R_{2323}-4R_{1414}}\\
=&-4\sum_{1\le i<j \le 4} R_{ijij}+2\left(R_{1414}+R_{2424}\right)+14\left(R_{1313}+R_{2323}\right)  \\
&+K(e_1+e_3,e_2+e_4)+K(e_1-e_3,e_2-e_4)+K(e_2+e_3,e_1-e_4)+K(e_2-e_3,e_1+e_4),
\end{align*}
where in the second equality,  we have used \eqref{sum1} and \eqref{sum2}. Thus,
\begin{align*}
12\left(R_{1313}+R_{2323}+R_{1234}\right)\ge&-2\left(n(n-1)R_0 -2\sum_{i=1}^4\sum_{j=5}^n R_{ijij} -\sum_{5\le i,j\le n}R_{ijij}\right) +48K_{min}\\
\ge&-2\left[n(n-1)R_0 -\left(2\times 4(n-4) +(n-4)(n-5)\right)K_{min}\right]+48K_{min}\\
=& 2\kh{-n(n-1)R_0 -\left( n(n-1) +12\right)K_{min}}.
\end{align*}

Hence, if
\begin{align*}
  K_{min}>\left(1-\dfrac{12}{n^2-n+12}\right)R_0,
\end{align*}
we obtain
$$
R_{1313}+R_{2323}+R_{1234}>0.
$$
Replace $e_4$ by $-e_4$, we obtain
\begin{align*}
R_{1313}+R_{2323}-R_{1234}>0.
\end{align*}
Hence,
\begin{align*}
R_{1313}+R_{2323}-\abs{R_{1234}}>0,
\end{align*}
Similarly,
\begin{align*}
R_{1414}+R_{2424}-\abs{R_{1234}}>0.
\end{align*}
Therefore,
$$
R_{1313}+R_{2323}+\lambda^2(R_{1414}+R_{2424})> (1+\lambda^2)\abs{R_{1234}}\ge 2\lambda R_{1234}.
$$
Our conclusion follows immediately from \autoref{diffsphere2}.
\end{proof}

When the dimension $n=4$, the following example indicates that our pinching constant is optimal.
\begin{exam}\label{exam}

 Consider  the Fubini-Study metric on $\Com P^n$, then we have
\begin{align*}
R(X,Y,X,Y)=1+3\abs{\hin{JX}{Y}}^2,
\end{align*}
for every orthonormal two-frame $\{X,Y\}$, where $J$ is the complex structure.
Let $n=2m$, consider a local orthonormal frame $\{e_1,\cdots,e_m, Je_1, \cdots, Je_m\}$, we have
\begin{align*}
  R(e_i,e_j,e_i,e_j)&=1, \quad \forall 1\le i\neq j\le m,\\
  R(e_i,Je_j,e_i,Je_j)&=1, \quad \forall 1\le i\neq j\le m,\\
  R(e_i,Je_i,e_i,Je_i)&=4, \quad \forall 1\le i\le m,\\
  R(Je_i,Je_j,Je_i,Je_j)&=1, \quad \forall 1\le i\neq j\le m
\end{align*}
Therefore,
\begin{align*}
  s=4m(m-1) +8m=n(n+2),\quad  Ric_M = 2m+2 = n+2, \quad K_{min} =1,\quad K_{max} =4,\\
  R_0=\frac{s}{n(n-1)} = \frac{n+2}{n-1}=\frac{Ric_M}{n-1},\ \ K_{min}=\frac{n-1}{n+2}R_0
  =\frac{n-1}{n+2}\frac{Ric_M}{n-1}, \ \ R_0 = \frac{Ric_M}{n-1} =\frac{n+2}{4(n-1)}K_{max}.
\end{align*}
When $n=4$, we have
$$
R_0 =\frac{Ric_M}{3} =  \frac{1}{2}K_{max}, \quad K_{min} = \frac{1}{2} R_0.
$$
\end{exam}

\begin{proof}[ Proof of \autoref{Kmax}]
We begin with the following identity:
\begin{align}\label{eq:new2-1}
n(n-1)R_0=&\sum_{i, j=5}^nR_{ijij}+2\sum_{i=1}^4\sum_{j=5}^nR_{ijij}+2\sum_{1\leq i<j\leq4}R_{ijij}.
\end{align}
Notice that, for $\lambda\in[0,1]$ and $\varepsilon>0$,
\begin{align}
&\sum_{1\leq i<j\leq4}R_{ijij}\label{eq:new2-2}\\
=&\sum_{1\leq i<j\leq4}R_{ijij}
-\frac{\varepsilon}{2(1+\lambda^2)}\left(\left(R_{1313}+R_{2323}\right)
+\lambda^2\left(R_{1414}+R_{2424}\right)-2\lambda R_{1234}\right)\notag\\
&+\frac{\varepsilon}{2(1+\lambda^2)}\left(\left(R_{1313}+R_{2323}\right)
+\lambda^2\left(R_{1414}+R_{2424}\right)-2\lambda R_{1234}\right)\notag
\\
=&R_{1212}+R_{3434}+
\left(1-\dfrac{\varepsilon}{2(1+\lambda^2)}\right)\left(R_{1313}
+R_{2323}\right)+\left(1-\dfrac{\varepsilon\lambda^2}{2(1+\lambda^2)}\right)
\left(R_{1414}+R_{2424}\right)+\dfrac{\varepsilon\lambda}{1+\lambda^2}R_{1234}\notag\\
&+\dfrac{\varepsilon}{2(1+\lambda^2)}\left(\left(R_{1313}+R_{2323}\right)
+\lambda^2\left(R_{1414}+R_{2424}\right)-2\lambda R_{1234}\right) .\notag
\end{align}
According to \autoref{lem2}, replace $e_4$ by $-e_4$, we obtain
\begin{align}\label{lem2a}
12R_{1234} = &-4\kh{R_{1212}+R_{3434}}-2\kh{R_{1313} +R_{1414} +R_{2323}+R_{2424}}\\
&+\kh{K(e_1+e_3, e_2+e_4) + K(e_1-e_3, e_2-e_4) + K(e_2+e_3, e_1-e_4) +K(e_2-e_3, e_1+e_4)}.\notag
\end{align}
Therefore, for fixed $\varepsilon_0 =12(\sqrt{10}-3)$,
\begin{align}
&R_{1212}+R_{3434}+
\left(1-\dfrac{\varepsilon_0}{2(1+\lambda^2)}\right)\left(R_{1313}
+R_{2323}\right)+\left(1-\dfrac{\varepsilon_0\lambda^2}{2(1+\lambda^2)}\right)
\left(R_{1414}+R_{2424}\right)+\dfrac{\varepsilon_0\lambda}{1+\lambda^2}R_{1234}\label{eq:new2-3}\\
=&\left(1-\dfrac{\varepsilon_0\lambda}{3(1+\lambda^2)}\right)\left(R_{1212}+R_{3434}\right)
+\left(1-\dfrac{\varepsilon_0(3+\lambda)}{6\left(1+\lambda^2\right)}\right)\left(R_{1313}+R_{2323}\right)
+\left(1-\dfrac{\varepsilon_0\kh{3\lambda^2+\lambda}}{6\left(1+\lambda^2\right)}\right)\left(R_{1414}+R_{2424}\right)\notag\\
&+\dfrac{\varepsilon_0\lambda\left(K(e_1+e_3, e_2+e_4) + K(e_1-e_3, e_2-e_4) + K(e_2+e_3, e_1-e_4) +K(e_2-e_3, e_1+e_4)\right)}{12(1+\lambda^2)}\notag\\
\le&\left(1-\dfrac{\varepsilon_0\lambda}{3(1+\lambda^2)}\right)\cdot 2K_{max}
+\left(1-\dfrac{\varepsilon_0(3+\lambda)}{6\left(1+\lambda^2\right)}\right)\cdot 2K_{max}
+\left(1-\dfrac{\varepsilon_0\kh{3\lambda^2+\lambda}}{6\left(1+\lambda^2\right)}\right)\cdot 2K_{max}
+\dfrac{16\varepsilon_0\lambda K_{max}}{12(1+\lambda^2)}\notag\\
=&(6-\varepsilon_0)K_{max},\notag
\end{align}
where we have used
$$
1-\dfrac{\varepsilon_0\lambda}{3(1+\lambda^2)}\ge 0, \quad
 1-\dfrac{\varepsilon_0(3+\lambda)}{6\left(1+\lambda^2\right)}\ge 0, \quad
1-\dfrac{\varepsilon_0\kh{3\lambda^2+\lambda}}{6\left(1+\lambda^2\right)}\ge 0,\quad \forall \lambda\in[0,1].
$$
Thus, \eqref{eq:new2-1}, \eqref{eq:new2-2} and \eqref{eq:new2-3} yield
\begin{align}
&\dfrac{\varepsilon_0}{1+\lambda^2}\left(\left(R_{1313}+R_{2323}\right)
+\lambda^2\left(R_{1414}+R_{2424}\right)-2\lambda R_{1234}\right)\label{eq:new2-4}\\
\geq& n(n-1)R_0 -(12-2\varepsilon_0)K_{max}-\kh{\sum_{i, j=5}^nR_{ijij}+2\sum_{i=1}^4\sum_{j=5}^nR_{ijij}}\notag\\
\ge& n(n-1)R_0 -(12-2\varepsilon_0)K_{max} -(n-4)(n-5)K_{max} - 8(n-4)K_{max}\notag\\
=& n(n-1)R_0 -(n^2-n-2\varepsilon_0)K_{max}.\notag
\end{align}
Consequently, the assumption $R_0 > \kh{1-\frac{2\varepsilon_0}{n(n-1)}}K_{max}$  combined with \eqref{eq:new2-4} imply
 $$
 \left(R_{1313}+R_{2323}\right)
+\lambda^2\left(R_{1414}+R_{2424}\right)-2\lambda R_{1234}>0.
 $$
Our conclusion follows from \autoref{diffsphere2} immediately.
\end{proof}
\begin{rem}\label{rem1}
  If we take $\lambda=1, \varepsilon_0=3$ in the above proof, we actually have, when
  \begin{align*}
    R_0 > \kh{1 -\frac{6}{n(n-1)}}K_{max},
  \end{align*}
  then the isotropic curvature
  \begin{align*}
    R_{1313}+R_{1414}+R_{2323}+R_{2424}>0,
  \end{align*}
  which implies $M$ is homeomorphic to a sphere.
  One can see from \autoref{exam}, the  pinching constant $\kh{1-\dfrac{6}{n(n-1)}}$
  is optimal when $n=4$.
\end{rem}

\begin{proof}[Proof of \autoref{2thRic}]
Let $D$ be a constant satisfying $Ric^{[4]}>4(n-1)D$. Then
\begin{equation}\label{eq:new3-1}
4(n-1)D<Ric_{11}+Ric_{22}+Ric_{33}+Ric_{44}=\sum_{i=1}^4\sum_{j=5}^nR_{ijij}+2\sum_{1\leq i<j\leq 4}R_{ijij}.
\end{equation}
Check the proof of \autoref{Kmax}, we actually have proved that for every $\lambda\in[0,1]$,
\begin{align*}
\sum_{1\leq i<j\leq 4}R_{ijij}\leq (6-\varepsilon_0)K_{max}+\dfrac{\varepsilon_0}{2\left(1+\lambda^2\right)}\left(\left(R_{1313}+R_{2323}\right)
+\lambda^2\left(R_{1414}+R_{2424}\right)-2\lambda R_{1234}\right),
\end{align*}
where $\varepsilon_0=12(\sqrt{10}-3)$.
Combined with \eqref{eq:new3-1}, we obtain
\begin{align*}
4(n-1)D<\left(4(n-1)-2\varepsilon_0\right)K_{max}+\dfrac{3}{1+\lambda^2}\left(\left(R_{1313}+R_{2323}\right)+\lambda^2\left(R_{1414}+R_{2424}\right)-2\lambda R_{1234}\right).
\end{align*}
Hence, if $$
\frac{Ric^{[4]}}{4(n-1)}>\kh{1-\frac{\varepsilon_0}{2(n-1)}}K_{max},
$$
 we have
\begin{align*}
\left(R_{1313}+R_{2323}\right)+\lambda^2\left(R_{1414}+R_{2424}\right)-2\lambda R_{1234}>0.
\end{align*}
We complete our proof.
\end{proof}
\begin{rem}\label{rem2}
Similar as \autoref{rem1}, one can take $\lambda=1$ and $\varepsilon_0 =3$ and obtain that, if
$$
\frac{Ric^{[4]}}{4(n-1)}>\kh{1-\frac{3}{2(n-1)}}K_{max},
$$
then $M$ has positive isotropic curvature, and is homeomorphic to a sphere.
\end{rem}

Moreover, if $M$ is Einstein, we obtain the following
\begin{cor}\label{Einstein}
   Let $M^n (n\ge 4)$ be a closed and simply connected Einstein manifold.  If
$$
R_0>\kh{1-\frac{3}{2(n-1)}}K_{max},
$$
then $M$ is isometric (by scaling) to $\Lg{S}^n$.
\end{cor}
\begin{proof}
If $M$ is Einstein, then $Ric =cg$ for some positive constant $c$, the normalized scalar
curvature $R_0=\frac{Ric^{[4]}}{4(n-1)}$. From \autoref{rem2}, we know the isotropic curvature
is positive. Therefore,
by Brendle's Theorem (\cite[Theorem 1]{Bre10}) we obtain the conclusion.
\end{proof}

\vspace{2ex}

\section{Submanifolds with pinching  curvatures}
In this section, we will prove some sphere theorems for a Riemannian manifold isometrically immersed into
another with some pinching curvature conditions. It is worth pointing out that our pinching constants
in this section also improve Gu-Xu's corresponding pinching constants in \cite{GuXu12} and \cite{XuGu14}.

Let $R$ denote an algebraic curvature tensor, for every orthonormal four-frame
$\set{e_1, e_2, e_3, e_4}$ and $\lambda,\mu\in[-1,1]$, we give the following notation,
\begin{align*}
\mathcal{I}_{\lambda,\mu}(R) = &\frac{1}{(1+\lambda^2)(1+\mu^2)}\kh{R_{1313}+\lambda^2R_{1414}+\mu^2R_{2323}+ \lambda^2\mu^2R_{2424} -2\lambda\mu R_{1234}},
\end{align*}
and we  denote $\mathcal{I}_{\lambda, 1}(R)$ briefly by $\mathcal{I}_\lambda (R)$.

Therefore, by Gauss equation \eqref{Gausseq}, we have
\begin{align}\label{algebraGauss}
  \mathcal{I}_\lambda(R) =\mathcal{I}_\lambda(\bar{R}^T) +\mathcal{I}_\lambda\kh{\frac{1}{2}B\circledwedge B}.
\end{align}
Corresponding \autoref{Kmax}, we have the following result:
\begin{theorem}\label{subkmax}
 Let $M^n$ be an $n(\ge 4)$-dimensional closed submanifold
in an $N$-dimensional Riemannian manifold $\bar{M}^N$. \\
(1)\quad If, pointwisely,
$$
\abs{B}^2<\frac{2N(N-1)}{3}\zkh{\bar{R}_0-\kh{1-\frac{6}{N(N-1)}}\bar{K}_{max}}+\frac{n^2\abs{H}^2}{n-2},
$$
then $M$ has positive isotropic curvature. Therefore, $\pi_k(M)=0$ for  $2 \le k \le \zkh{\frac{n}{2}}$.
In particular,
if $M$ is simply connected, then $M$ is homeomorphic to a sphere.\\
(2) \quad If, pointwisely,
$$
\abs{B}^2<\frac{N(N-1)}{3}\zkh{\bar{R}_0-\kh{1-\frac{24(\sqrt{10}-3)}{N(N-1)}}\bar{K}_{max}}+\frac{n^2\abs{H}^2}{n-1},
$$
then $M$ is diffeomorphic to a spherical space form.
In particular,
if $M$ is simply connected, then $M$ is diffeomorphic to $\Lg{S}^n$.

\end{theorem}
\begin{proof}
Let $\bar D$ be a constant satisfying $N(N-1)\bar D<\sum_{i,j=1}^N \bar{R}_{ijij}$.
Then a similar algebraic argument as the proof of \autoref{Kmax} gives a
similar inequality as \eqref{eq:new2-4}:
\begin{align}\label{isobarR}
6\mathcal{I}_\lambda\kh{\bar{R}^T} > N(N-1)\bar D -(N^2-N-2\varepsilon_0)\bar{K}_{max}.
\end{align}
For $\lambda=1$, we take $\varepsilon_0=3$.
In the proof of \cite[Lemma 9]{GuXu12}, the authors give the following estimate
\begin{align}\label{isoB1}
4\mathcal{I}_1\kh{\frac{1}{2}B\circledwedge B}\ge \frac{n^2H^2}{n-2} -\abs{B}^2.
\end{align}
Thus, \eqref{algebraGauss}, \eqref{isobarR} and \eqref{isoB1} yield
\begin{align*}
  \mathcal{I}_1(R)
  >&\frac{1}{6}\kh{N(N-1)\bar D -(N^2-N-6)\bar{K}_{max}} + \frac14 \kh{\frac{n^2\abs{H}^2}{n-2}-\abs{B}^2}\\
  = &\frac14\set{\frac{2N(N-1)}{3}\zkh{\bar{D}-\kh{1-\frac{6}{N(N-1)}}\bar{K}_{max}}+\frac{n^2\abs{H}^2}{n-2} -\abs{B}^2}.
\end{align*}
Combined with \autoref{toposphere}, we complete the proof of Claim (1).

For arbitrary $\lambda\in[0,1]$, we take $\varepsilon_0=12(\sqrt{10}-3)$.
In the proof of \cite[Lemma 11]{GuXu12}, the authors obtain
\begin{align}\label{isoB}
2\mathcal{I}_\lambda\kh{\frac{1}{2}B\circledwedge B} \ge \frac{n^2\abs{H}^2}{n-1} -\abs{B}^2, \quad \forall \lambda\in[0,1]
\end{align}
Thus, \eqref{algebraGauss}, \eqref{isobarR} and \eqref{isoB} give
\begin{align*}
  \mathcal{I}_\lambda(R)
  >&\frac{1}{6}\kh{N(N-1)\bar D -(N^2-N-24(\sqrt{10}-3))\bar{K}_{max}} + \frac12 \kh{\frac{n^2\abs{H}^2}{n-1}-\abs{B}^2}\\
  = &\frac12\set{\frac{N(N-1)}{3}\zkh{\bar{D}-\kh{1-\frac{24(\sqrt{10}-3)}{N(N-1)}}\bar{K}_{max}}+\frac{n^2\abs{H}^2}{n-1} -\abs{B}^2}.
\end{align*}
Then Claim (2) follows easily from \autoref{diffsphere2}.
\end{proof}

After a similar argument we also have the following two extrinsic sphere theorems corresponding to \autoref{Kmin} and \autoref{2thRic}.

\begin{theorem}\label{subkmin}
    Let $M^n$ be an $n(\ge 4)$-dimensional closed submanifold
in an $N$-dimensional Riemannian manifold $\bar{M}^N$. \\
(1)\quad If, pointwisely,
$$
\abs{B}^2<\frac{N^2-N+12}{3}\kh{ \bar{K}_{min}-\left(1-\dfrac{12}{N^2-N+12}\right)\bar{R}_0}+\frac{n^2\abs{H}^2}{n-2},
$$
then $M$ has positive isotropic curvature. Therefore, $\pi_k(M)=0$ for  $2 \le k \le \zkh{\frac{n}{2}}$.
In particular,
if $M$ is simply connected, then $M$ is homeomorphic to a sphere.\\
(2) \quad If, pointwisely,
$$
\abs{B}^2<\frac{N^2-N+12}{6}\kh{ \bar{K}_{min}-\left(1-\dfrac{12}{N^2-N+12}\right)\bar{R}_0}+\frac{n^2\abs{H}^2}{n-1},
$$
then $M$ is diffeomorphic to a spherical space form.
In particular,
if $M$ is simply connected, then $M$ is diffeomorphic to $\Lg{S}^n$.

\end{theorem}

\begin{theorem}\label{subric}
    Let $M^n$ be an $n(\ge 4)$-dimensional closed submanifold
in an $N$-dimensional Riemannian manifold $\bar{M}^N$. \\
(1)\quad If, pointwisely,
$$
\abs{B}^2 < \frac{8(N-1)}{3}\kh{\frac{\overline{Ric}^{[4]}_{min}}{4(N-1)}-\kh{1-\frac{3}{2(N-1)}}\bar{K}_{max}} +\frac{n^2H^2}{n-2},
$$
then $M$ has positive isotropic curvature. Therefore, $\pi_k(M)=0$ for  $2 \le k \le \zkh{\frac{n}{2}}$.
In particular,
if $M$ is simply connected, then $M$ is homeomorphic to a sphere.\\
(2) \quad If, pointwisely,
$$
\abs{B}^2 < \frac{4(N-1)}{3}\kh{\frac{\overline{Ric}^{[4]}_{min}}{4(N-1)}-\kh{1-\frac{6(\sqrt{10}-3)}{N-1}}\bar{K}_{max}} +\frac{n^2H^2}{n-1},
$$
then $M$ is diffeomorphic to a spherical space form.
In particular,
if $M$ is simply connected, then $M$ is diffeomorphic to $\Lg{S}^n$.
\end{theorem}

Also we have the following corollary corresponding to \autoref{Einstein}.
\begin{cor}
  Let $M^n$ be an $n(\ge 4)$-dimensional closed Einstein submanifold
in an $N$-dimensional Riemannian manifold $\bar{M}^N$.  If, pointwisely,
$$
\abs{B}^2 < \frac{8(N-1)}{3}\kh{\frac{\overline{Ric}^{[4]}_{min}}{4(N-1)}-\kh{1-\frac{3}{2(N-1)}}\bar{K}_{max}} +\frac{n^2H^2}{n-2},
$$
then $M$ is isometric to a spherical space form. In particular,
if $M$ is simply connected, then $M$ is isometric to $\Lg{S}^n$ (by scaling).
\end{cor}

\begin{rem}
  Using a similar method, we also can get a sphere theorem under pinched curvature by $K_{min}$. But since the pinching constant
  is the same as Gu-Xu's result in \cite{GuXu12}, we omit here.
\end{rem}

\vspace{2ex}

Next we will use a complex orthonormal frame to state the proofs of \autoref{subR0}, \autoref{subric1} and
\autoref{subric2}. One can verify that in suitable complex orthonormal frame, the calculations will be considerably simplified.

\begin{proof}[Proof of \autoref{subR0}]
  Let $e_1, \cdots, e_n$ be a local orthonormal frame of $TM$.
For $\lambda, \mu\in[0,1]$, define
$$
\varepsilon_1=\dfrac{e_1+\sqrt{-1}\lambda e_2}{\sqrt{1+\lambda^2}},\quad\varepsilon_2=\dfrac{e_3+\sqrt{-1}\mu e_4}{\sqrt{1+\mu^2}},
$$
and extend these two vectors to be a local orthonormal frame of
$T^{\mathbb C}M$.
Then a direct computation gives
$$
R_{12\bar1\bar2} = R(\varepsilon_1, \varepsilon_2, \bar{\varepsilon}_1, \bar{\varepsilon}_2) =\mathcal{I}_{\lambda, \mu} (R).
$$

We first claim that
\begin{align}\label{eq:4-2}
\sum_{i,j=1}^n\bar{R}_{ij\,\bar{i}\bar{j}}\le (n^2-n-2)\bar{K}_{max}+2\bar{R}_{12\bar1\bar2}.
 \end{align}

If this is true, then
\eqref{complexRicci} and \eqref{eq:4-2} give
\begin{align*}
n(n-1)R_0 =& \sum_{i,j=1}^n R_{ij\,\bar{i}\bar{j}}\\
=&\sum_{i,j=1}^n\bar{R}_{ij\,\bar{i}\bar{j}} +n(n-1)\abs{H}^2
-\sum_{i,j=1}^n\sum_\alpha \abs{\mathring{h}^\alpha_{i \bar{j}}}^2\\
\le & (n^2-n-2)\bar{K}_{max} + 2\bar{R}_{12\bar1\bar2} +n(n-1)\abs{H}^2
-\sum_{i,j=1}^n\sum_\alpha \abs{\mathring{h}^\alpha_{i \bar{j}}}^2\\
=& (n^2-n-2)\bar{K}_{max} + 2 R_{12\bar1\bar2} +n(n-1)\abs{H}^2\\
&-2\kh{\abs{H}^2+ \sum_\alpha\kh{H^{\alpha}\left(\mathring{h}^{\alpha}_{1\bar1}
+\mathring{h}^{\alpha}_{2\bar 2}\right)+\mathring{h}^{\alpha}_{1\bar 1}\mathring{h}^{\alpha}_{2\bar 2}
-\abs{\mathring{h}^{\alpha}_{1\bar 2}}^2}}-\sum_{i,j=1}^n\sum_\alpha \abs{\mathring{h}^\alpha_{i \bar{j}}}^2.
\end{align*}
On the other hand,
\begin{align*}
  &-2\kh{\abs{H}^2+ \sum_\alpha\kh{H^{\alpha}\left(\mathring{h}^{\alpha}_{1\bar1}
+\mathring{h}^{\alpha}_{2\bar 2}\right)+\mathring{h}^{\alpha}_{1\bar 1}\mathring{h}^{\alpha}_{2\bar 2}
-\abs{\mathring{h}^{\alpha}_{1\bar 2}}^2}}-\sum_{i,j=1}^n\sum_\alpha \abs{\mathring{h}^\alpha_{i \bar{j}}}^2\\
\le&-2\kh{\abs{H}^2+ \sum_\alpha\kh{H^{\alpha}\left(\mathring{h}^{\alpha}_{1\bar1}
+\mathring{h}^{\alpha}_{2\bar 2}\right)+\mathring{h}^{\alpha}_{1\bar 1}\mathring{h}^{\alpha}_{2\bar 2}
-\abs{\mathring{h}^{\alpha}_{1\bar 2}}^2}}-\sum_\alpha\kh{\abs{\mathring{h}^{\alpha}_{1\bar 1}}^2
+\abs{\mathring{h}^{\alpha}_{2\bar 2}}^2+2\abs{\mathring{h}^{\alpha}_{1\bar 2}}^2
+\sum_{i,j=3}^n \abs{\mathring{h}^\alpha_{i \bar{j}}}^2}\\
\le & -2 \abs{H}^2 -2 \sum_\alpha H^{\alpha}\left(\mathring{h}^{\alpha}_{1\bar1}
+\mathring{h}^{\alpha}_{2\bar 2}\right)
-\sum_\alpha \kh{\abs{\mathring{h}^\alpha_{1 \bar{1}}}^2 +\abs{\mathring{h}^\alpha_{2 \bar{2}}}^2
+2\mathring{h}^{\alpha}_{1\bar 1}\mathring{h}^{\alpha}_{2\bar 2}}
-\frac{1}{n-2}\sum_\alpha\kh{\mathring{h}^\alpha_{1\bar1}+\mathring{h}^\alpha_{2\bar2}}^2\\
= & -2 \abs{H}^2 -2 \sum_\alpha H^{\alpha}\left(\mathring{h}^{\alpha}_{1\bar1}
+\mathring{h}^{\alpha}_{2\bar 2}\right)
-\kh{1+\frac{1}{n-2}}\sum_\alpha\kh{\mathring{h}^\alpha_{1\bar1}+\mathring{h}^\alpha_{2\bar2}}^2\\
\le & -\frac{n}{n-1}\abs{H}^2,
\end{align*}
where in the second inequality, we have used
\begin{align*}
 \sum_{i,j=3}^n\abs{\mathring{h}^{\alpha}_{i\bar j}}^2
 \ge \sum_{i=3}^n \abs{\mathring{h}^{\alpha}_{i\bar i}}^2
 \ge \frac{\kh{\sum_{i=3}^n \mathring{h}^{\alpha}_{i\bar i}}^2}{n-2}
 =\frac{\kh{\mathring{h}^{\alpha}_{1\bar 1}+\mathring{h}^{\alpha}_{2\bar 2}}^2}{n-2}.
\end{align*}
Therefore, we have
\begin{align*}
n(n-1)R_0\leq
  (n^2-n-2)\bar{K}_{max} + 2 R_{12\bar1\bar2} +\kh{n(n-1)-\frac{n}{n-1}}\abs{H}^2,
\end{align*}
which implies
\begin{align}\label{eq:4-3}
2 R_{12\bar1\bar2}\geq n(n-1)\zkh{R_0
 - \kh{\kh{1-\frac{2}{n(n-1)}}\bar{K}_{max}   +\frac{n(n-2)}{(n-1)^2}\abs{H}^2}}.
\end{align}
Thus, by the assumption of this theorem and \eqref{eq:4-3}, we have
$R_{12\bar1\bar2}\geq0$. Therefore, $M\times\R^2$ has nonnegative isotropic curvature (see for example \cite[Proposition 17.8]{Bre10a}). Also by the assumption,  the isotropic curvature of $M\times\R^2$ is positive at some point. Consequently, $M$ has nonnegative isotropic curvature
and positive isotropic curvature at some point. Then   $M$  admits a metric with positive
isotropic curvature (see \cite{Ses09}). Therefore, $M$ is a topological sphere by \autoref{toposphere}. But by the classification theorem of Brendle-Schoen (\autoref{classif}), $M$ must be diffeomorphic to
$\Lg{S}^n$.

It remains to prove the inequality \eqref{eq:4-2}. Under the orthonormal frames $\set{e_i}$, this inequality is equivalent to
\begin{align*}
\sum_{i,j=1}^n\bar R_{ijij}\leq(n^2-n-2)\bar K_{max}+2\mathcal{I}_{\lambda,\mu}\kh{\bar{R}^T}.
\end{align*}
Notice that
\begin{align*}
\sum_{i,j=1}^n\bar R_{ijij}=&2\mathcal{I}_{\lambda,\mu}\kh{\bar{R}^T}+2\kh{\sum_{1\le i<j\le 4}\bar{R}_{ijij}-\mathcal{I}_{\lambda,\mu}\kh{\bar{R}^T}}
  +2\sum_{i=1}^4\sum_{j=5}^n \bar{R}_{ijij} +\sum_{i,j=5}^n \bar{R}_{ijij}\\
  \leq&2\mathcal{I}_{\lambda,\mu}\kh{\bar{R}^T}+2\kh{\sum_{1\le i<j\le 4}\bar{R}_{ijij}-\mathcal{I}_{\lambda,\mu}\kh{\bar{R}^T}}
  +(n^2-n-12)\bar K_{max}.
\end{align*}
Therefore, it is sufficient to prove
\begin{align*}
\sum_{1\le i<j\le 4}\bar{R}_{ijij}-\mathcal{I}_{\lambda,\mu}\kh{\bar{R}^T}\leq 5\bar K_{max}.
\end{align*}
A direct computation using
 \eqref{lem2a} yields
 \begin{align*}
   &\sum_{1\le i<j\le 4}\bar{R}_{ijij}-\mathcal{I}_{\lambda, \mu}(\bar{R}^T)\\
   =& \sum_{1\le i<j\le 4}\bar{R}_{ijij}
   - \frac{1}{(1+\lambda^2)(1+\mu^2)}\kh{\bar R_{1313}+\lambda^2\bar R_{1414}+\mu^2\bar R_{2323}+ \lambda^2\mu^2\bar R_{2424} -2\lambda\mu \bar R_{1234}}\\
   =&\kh{1-\frac{2\lambda\mu}{3(1+\lambda^2)(1+\mu^2)}}\kh{\bar{R}_{1212}+\bar{R}_{3434}} +\kh{1-\frac{3+\lambda\mu}{3(1+\lambda^2)(1+\mu^2)}}\bar{R}_{1313}\\
   &+\kh{1-\frac{3\lambda^2+\lambda\mu}{3(1+\lambda^2)(1+\mu^2)}}\bar{R}_{1414}
   +\kh{1-\frac{3\mu^2+\lambda\mu}{3(1+\lambda^2)(1+\mu^2)}}\bar{R}_{2323}
   +\kh{1-\frac{3\lambda^2\mu^2+\lambda\mu}{3(1+\lambda^2)(1+\mu^2)}}\bar{R}_{2424}\\
   & +\dfrac{\lambda\mu\left(\bar K(e_1+e_3, e_2+e_4) +\bar K(e_1-e_3, e_2-e_4)
+ \bar K(e_2+e_3, e_1-e_4) +\bar K(e_2-e_3, e_1+e_4)\right)}{6(1+\lambda^2)(1+\mu^2)}\\
\le &\kh{1-\frac{2\lambda\mu}{3(1+\lambda^2)(1+\mu^2)}}\cdot 2\bar K_{max} +\kh{1-\frac{3+\lambda\mu}{3(1+\lambda^2)(1+\mu^2)}}\bar K_{max}
+\kh{1-\frac{3\lambda^2+\lambda\mu}{3(1+\lambda^2)(1+\mu^2)}}\bar K_{max}
\\
&+\kh{1-\frac{3\mu^2+\lambda\mu}{3(1+\lambda^2)(1+\mu^2)}}\bar K_{max}
   +\kh{1-\frac{3\lambda^2\mu^2+\lambda\mu}{3(1+\lambda^2)(1+\mu^2)}}\bar K_{max}
   +\dfrac{16\lambda\mu}{6(1+\lambda^2)(1+\mu^2)}\bar K_{max}\\
=&5\bar K_{max}.
 \end{align*}
\end{proof}

\autoref{subric1} and \autoref{subric2} are easy consequences of the following theorem:

\begin{theorem}\label{epsilon}
For fixed $0<\varepsilon\le 1$,  set
$\delta(\varepsilon, n) = \frac{\kh{(n-4)\varepsilon+2}^2n^2}{8\kh{2+\kh{n^2-4n+2}\varepsilon}}$.
Suppose $M^n (n\geq4)$ is a closed simply connected submanifold of $\bar{M}^N$ satisfying
\begin{align*}
\frac{Ric^{[2]}}{2}\geq(n-1-\varepsilon)\bar K_{max}+\delta(\varepsilon, n)\abs{H}^2,
\end{align*}
with strict inequality at some point,
then $M$ is   diffeomorphic to $\Lg{S}^n$.
 \end{theorem}

\begin{proof}
Let $\set{e_i}$ be  a local orthonormal frame of $TM$. For $\lambda,\mu\in[0,1]$, define
\begin{align*}
\varepsilon_1=\dfrac{e_1+\sqrt{-1}\mu e_2}{\sqrt{1+\mu^2}},\quad\varepsilon_2=\dfrac{e_3+\sqrt{-1}\lambda e_4}{\sqrt{1+\lambda^2}},\quad
\varepsilon_3 = \dfrac{\mu e_1-\sqrt{-1}e_2}{\sqrt{1+\mu^2}},\quad  \varepsilon_4=\frac{\lambda e_3-\sqrt{-1}e_4}{\sqrt{1+\lambda^2}}, \\
 \varepsilon_i = e_i,  \ \ 5\le i\le n.
\end{align*}
Then $\set{\varepsilon_i}$ is a local orthonormal frame of $T^{\Com}M$.
Similar as the proof of \autoref{subR0}, it is sufficient to
prove $R_{12\bar1\bar2}\geq0$ and the strict inequality holds for all frame $\set{e_i}$ and all numbers $\lambda, \mu\in[0,1]$ at some point.
Ricci curvature formula  \eqref{complexRicci} gives
\begin{align}\label{ric1}
&\frac{1}{2}\kh{Ric_{1\bar1}+Ric_{2\bar2}}\\
=&\frac{1}{2}\kh{\sum_{i\neq 1}\bar{R}_{1i\bar1\bar{i}}+\sum_{i\neq 2}\bar{R}_{2i\bar2\bar{i}}}+(n-1)\abs{H}^2+\frac{1}{2}\sum_\alpha\kh{(n-2)H^{\alpha}\kh{\mathring{h}^{\alpha}_{1\bar1}+
\mathring{h}^\alpha_{2\bar2}}-\sum_{i=1}^n\kh{\abs{\mathring{h}^{\alpha}_{1\bar i}}^2+\abs{\mathring{h}^{\alpha}_{2\bar i}}^2}},\notag
\\\label{ric2}
&\frac{1}{n-2}\sum_{i= 3}^n Ric_{i\bar i}\\
=&\frac{1}{n-2}\zkh{\sum_{i=3}^n\sum_{j\neq i}\bar{R}_{ij\,\bar{i}\bar{j}}
-\sum_\alpha\kh{(n-2)H^{\alpha}\kh{\mathring{h}^{\alpha}_{1\bar1}+\mathring{h}^\alpha_{2\bar2}}
+\sum_{i=3}^n\sum_{j=1}^n\abs{\mathring{h}^{\alpha}_{i\bar j}}^2}}+(n-1)\abs{H}^2.\notag
\end{align}
Assume
\begin{align*}
Ric_{i\bar i}+Ric_{j\bar j}\geq 2D,\quad\forall 1\leq i<j\leq n.
\end{align*}
Then
\begin{align*}
  D\leq\frac{Ric_{1\bar1} + Ric_{2\bar2}}{2},\quad
  D\leq\frac{\sum_{3\le j<k\le n}\kh{Ric_{j\bar{j}} +Ric_{k\bar{k}}}}{(n-2)(n-3)} =\frac{\sum_{i=3}^n Ric_{i\bar{i}}}{n-2}.
\end{align*}
Hence for every $0<\varepsilon\le1$, by using \eqref{ric1}, \eqref{ric2} and \eqref{complexGauss}, we get
\begin{align*}
  D\leq&\ \  \varepsilon\cdot \frac{Ric_{1\bar1} + Ric_{2\bar2}}{2} + (1-\varepsilon)\cdot\frac{\sum_{i=3}^n Ric_{i\bar{i}}}{n-2}\\
=&\ \frac{\varepsilon}{2}\sum_{i=1}^n\kh{\bar{R}_{1i\bar1\bar{i}}+\bar{R}_{2i\bar2\bar{i}}}
+\frac{1-\varepsilon}{n-2}\sum_{i=3}^n\sum_{j=1}^n\bar{R}_{ij\,\bar{i}\bar{j}}+(n-1)\abs{H}^2
-\frac{1-\varepsilon}{n-2}\sum_{i=3}^n\sum_{j=1}^n\abs{\mathring{h}^{\alpha}_{i\bar j}}^2\\
&+\sum_\alpha\zkh{\frac{n\varepsilon -2}{2}H^\alpha\kh{\mathring{h}^\alpha_{1\bar1}+\mathring{h}^\alpha_{2\bar2}}
- \frac{\varepsilon}{2}
\sum_{i=1}^n\kh{\abs{\mathring{h}^{\alpha}_{1\bar i}}^2+\abs{\mathring{h}^{\alpha}_{2\bar i}}^2}}\\
=&\ \frac{\varepsilon}{2}\sum_{i=3}^n\kh{\bar{R}_{1i\bar1\bar{i}}+\bar{R}_{2i\bar2\bar{i}}}
+\frac{1-\varepsilon}{n-2}\sum_{i=3}^n\sum_{j=1}^n\bar{R}_{ij\,\bar{i}\bar{j}}+\varepsilon R_{12\bar1\bar2}
+(n-1-\varepsilon)\abs{H}^2
-\frac{1-\varepsilon}{n-2}\sum_{i=3}^n\sum_{j=1}^n\abs{\mathring{h}^{\alpha}_{i\bar j}}^2\\
&+\sum_\alpha\zkh{\frac{(n-2)\varepsilon -2}{2}H^\alpha\kh{\mathring{h}^\alpha_{1\bar1}
+\mathring{h}^\alpha_{2\bar2}}-\frac{\varepsilon}{2}\kh{\abs{\mathring{h}^\alpha_{1\bar1}}^2
+2\mathring{h}^\alpha_{1\bar1}\mathring{h}^\alpha_{2\bar2}
+\abs{\mathring{h}^\alpha_{2\bar2}}^2}
- \frac{\varepsilon}{2}
\sum_{i=3}^n\kh{\abs{\mathring{h}^{\alpha}_{1\bar i}}^2+\abs{\mathring{h}^{\alpha}_{2\bar i}}^2}}\\
\le&\ \frac{\varepsilon}{2}\sum_{i=3}^n\kh{\bar{R}_{1i\bar1\bar{i}}+\bar{R}_{2i\bar2\bar{i}}}
+\frac{1-\varepsilon}{n-2}\sum_{i=3}^n\sum_{j=1}^n\bar{R}_{ij\,\bar{i}\bar{j}}+\varepsilon R_{12\bar1\bar2}
+(n-1-\varepsilon)\abs{H}^2\\
&+\sum_\alpha\zkh{\frac{(n-2)\varepsilon -2}{2}H^\alpha\kh{\mathring{h}^\alpha_{1\bar1}
+\mathring{h}^\alpha_{2\bar2}}-\kh{\frac{\varepsilon}{2}+\frac{1-\varepsilon}{(n-2)^2}}\kh{\mathring{h}^\alpha_{1\bar1}+\mathring{h}^\alpha_{2\bar2}}^2}\\
\le&\ \frac{\varepsilon}{2}\sum_{i=3}^n\kh{\bar{R}_{1i\bar1\bar{i}}+\bar{R}_{2i\bar2\bar{i}}}
+\frac{1-\varepsilon}{n-2}\sum_{i=3}^n\sum_{j=1}^n\bar{R}_{ij\,\bar{i}\bar{j}}+\varepsilon R_{12\bar1\bar2}
+\delta(\varepsilon,n)\abs{H}^2
\end{align*}
where in the second inequality, we have used
\begin{align*}
 \sum_{i=3}^n\sum_{j=1}^n\abs{\mathring{h}^{\alpha}_{i\bar j}}^2
 \ge \sum_{i=3}^n \abs{\mathring{h}^{\alpha}_{i\bar i}}^2
 \ge \frac{\kh{\sum_{i=3}^n \mathring{h}^{\alpha}_{i\bar i}}^2}{n-2}
 =\frac{\kh{\mathring{h}^{\alpha}_{1\bar 1}+\mathring{h}^{\alpha}_{2\bar 2}}^2}{n-2}.
\end{align*}
Therefore,
\begin{align}\label{D}
  \varepsilon R_{12\bar1\bar2}\geq D -\kh{\frac{\varepsilon}{2}\sum_{i=3}^n\kh{\bar{R}_{1i\bar1\bar{i}}+\bar{R}_{2i\bar2\bar{i}}}
+\frac{1-\varepsilon}{n-2}\sum_{i=3}^n\sum_{j=1}^n\bar{R}_{ij\,\bar{i}\bar{j}}
+\delta(\varepsilon, n)\abs{H}^2}.
\end{align}

We  claim that
\begin{align}\label{sum3}
\frac{\varepsilon}{2}\sum_{i=3}^n\kh{\bar{R}_{1i\bar1\bar{i}}+\bar{R}_{2i\bar2\bar{i}}}
+\frac{1-\varepsilon}{n-2}\sum_{i=3}^n\sum_{j=1}^n\bar{R}_{ij\,\bar{i}\bar{j}}
\le (n-1-\varepsilon)\bar K_{max}.
\end{align}
If this is true, then combined with \eqref{D}, we have
\begin{align}\label{D1}
  \varepsilon R_{12\bar1\bar2} \geq D -(n-1-\varepsilon)\bar K_{max}
+\delta(\varepsilon, n)\abs{H}^2.
\end{align}

By the assumption of the theorem, we get
\begin{align*}
Ric(\varepsilon_1,\bar\varepsilon_1)+Ric(\varepsilon_2, \bar\varepsilon_2)=&\dfrac{Ric(e_1,e_1)+\mu^2Ric(e_2,e_2)}{1+\mu^2}+\dfrac{Ric(e_3,e_3)+\lambda^2	Ric(e_4,e_4)}{1+\lambda^2}\\
=&\dfrac{\left(Ric(e_1,e_1)+Ric(e_3,e_3)\right)+\lambda^2\left(Ric(e_1,e_1)+Ric(e_4,e_4)\right)}{\left(1+\lambda^2\right)\left(1+\mu^2\right)}\\
&+\dfrac{\mu^2\left(Ric(e_2,e_2)+Ric(e_3,e_3)\right)+\lambda^2\mu^2\left(Ric(e_2,e_2)+Ric(e_4,e_4)\right)}{\left(1+\lambda^2\right)\left(1+\mu^2\right)}\\
\geq&2\kh{(n-1-\varepsilon)\bar K_{max}+\delta(\varepsilon, n)\abs{H}^2}.
\end{align*}
Therefore, by the arbitrariness of  $e_1, e_2, e_3, e_4$, we can take
\begin{align*}
D=(n-1-\varepsilon)\bar K_{max}+\delta(\varepsilon, n)\abs{H}^2.
\end{align*}
Combining the above inequality with \eqref{D1}, we have
\begin{align*}
\varepsilon R(\varepsilon_1,\varepsilon_2,\bar\varepsilon_1,\bar\varepsilon_2)\geq D -\kh{(n-1-\varepsilon)\bar K_{max}+\delta(\varepsilon, n)\abs{H}^2}=0.
\end{align*}
Therefore we have $R_{12\bar1\bar2}\ge 0$, and strict inequality holds for all frame $\set{e_i}$ and all numbers $\lambda, \mu\in[0,1]$ at some point.

What is left is to prove the inequality \eqref{sum3}.
Under the given basis of $T^{\mathbb C}M$,
a direct computation  gives
\begin{align}
\label{eq:4-111} \bar{R}(\varepsilon_1,\varepsilon_2,\bar\varepsilon_1,\bar\varepsilon_2)
=&\dfrac{\bar R_{1313}+\mu^2\bar R_{2323}+\lambda^2\bar R_{1414}+\lambda^2\mu^2\bar R_{2424}-2\lambda\mu\bar R_{1234}}{\left(1+\lambda^2\right)\left(1+\mu^2\right)}\\
\label{eq:4-211}
\sum_{i=1}^n\kh{\bar{R}(\varepsilon_1,\varepsilon_i,
\bar\varepsilon_1,\bar\varepsilon_i)+\bar{R}(\varepsilon_2,
\varepsilon_i,\bar\varepsilon_2,\bar\varepsilon_i)}=&
\sum_{i=1}^n\zkh{\dfrac{\bar{R}_{1i1i}+\mu^2\bar{R}_{2i2i}}{1+\mu^2}+\dfrac{\bar R_{3i3i}+\lambda^2\bar R_{4i4i}}{1+\lambda^2}},\\
\sum_{i=1}^n\kh{\bar{R}(\varepsilon_3,\varepsilon_i,
\bar\varepsilon_3,\bar\varepsilon_i)+\bar{R}(\varepsilon_4,
\varepsilon_i,\bar\varepsilon_4,\bar\varepsilon_i)}
=&\sum_{i=1}^n\sum_{i=1}^n\zkh{\dfrac{\mu^2\bar{R}_{1i1i}
+\bar{R}_{2i2i}}{1+\mu^2}+\dfrac{\lambda^2\bar R_{3i3i}
+\bar R_{4i4i}}{1+\lambda^2}},\notag\\
\sum_{j=1}^n\bar{R}(\varepsilon_i,\varepsilon_j,
\bar\varepsilon_i,\bar\varepsilon_j)=&\sum_{j=1}^n
\bar R_{ijij},\quad 5\le i\le n.\notag
\end{align}

Note that,
\begin{align*}
&\dfrac{\varepsilon}{2}\sum_{i=3}^n\kh{\bar{R}(\varepsilon_1,\varepsilon_i,\bar\varepsilon_1,\bar\varepsilon_i)+\bar{R}(\varepsilon_2,\varepsilon_i,\bar\varepsilon_2,\bar\varepsilon_i)}
+\dfrac{1-\varepsilon}{n-2}\sum_{i=3}^n\sum_{j=1}^n\bar{R}(\varepsilon_i,\varepsilon_j,\bar\varepsilon_i,\bar\varepsilon_j)\\
=&\dfrac{\varepsilon}{2}\sum_{i=1}^n\kh{\bar{R}(\varepsilon_1,\varepsilon_i,\bar\varepsilon_1,\bar\varepsilon_i)+\bar{R}(\varepsilon_2,\varepsilon_i,\bar\varepsilon_2,\bar\varepsilon_i)}-\varepsilon\bar{R}(\varepsilon_1,\varepsilon_2,\bar\varepsilon_1,\bar\varepsilon_2)
+\dfrac{1-\varepsilon}{n-2}\sum_{i=3}^n\sum_{j=1}^n\bar{R}(\varepsilon_i,\varepsilon_j,\bar\varepsilon_i,\bar\varepsilon_j)\\
\leq&\frac{\varepsilon}{2}\sum_{i=1}^n\kh{\bar{R}(\varepsilon_1,\varepsilon_i,\bar\varepsilon_1,\bar\varepsilon_i)+\bar{R}(\varepsilon_2,\varepsilon_i,\bar\varepsilon_2,\bar\varepsilon_i)}-\varepsilon\bar{R}(\varepsilon_1,\varepsilon_2,\bar\varepsilon_1,\bar\varepsilon_2)
+\frac{1-\varepsilon}{n-2}\cdot(n-2)(n-1)\bar K_{max}.
\end{align*}
By using \eqref{eq:4-111} and \eqref{eq:4-211}, we have

\begin{align*}
&\frac{1}{2}\sum_{i=1}^n\kh{\bar{R}
(\varepsilon_1,\varepsilon_i,\bar\varepsilon_1,
\bar\varepsilon_i)+\bar{R}(\varepsilon_2,\varepsilon_i,
\bar\varepsilon_2,\bar\varepsilon_i)}-\bar{R}
(\varepsilon_1,\varepsilon_2,\bar\varepsilon_1,\bar\varepsilon_2)\\
=
&\frac{1}{2}\sum_{i=1}^4\zkh{\frac{\bar{R}_{1i1i}+\mu^2\bar{R}_{2i2i}}{1+\mu^2}
+\frac{\bar{R}_{3i3i}+\lambda^2\bar{R}_{4i4i}}{1+\lambda^2}}
+\frac{1}{2}\sum_{i=5}^n\zkh{\frac{\bar{R}_{1i1i}+\mu^2\bar{R}_{2i2i}}{1+\mu^2}
+\frac{\bar{R}_{3i3i}+\lambda^2\bar{R}_{4i4i}}{1+\lambda^2}}\\
&-\dfrac{\bar R_{1313}+\mu^2\bar R_{2323}+\lambda^2\bar R_{1414}+\lambda^2\mu^2\bar R_{2424}-2\lambda\mu\bar R_{1234}}{\left(1+\lambda^2\right)\left(1+\mu^2\right)}\\
\le
&\frac{1}{2}\sum_{i=1}^4\zkh{\frac{\bar{R}_{1i1i}+\mu^2\bar{R}_{2i2i}}{1+\mu^2}
+\frac{\bar{R}_{3i3i}+\lambda^2\bar{R}_{4i4i}}{1+\lambda^2}}
+(n-4)\bar{K}_{max}\\
&-\dfrac{\bar R_{1313}+\mu^2\bar R_{2323}+\lambda^2\bar R_{1414}+\lambda^2\mu^2\bar R_{2424}-2\lambda\mu\bar R_{1234}}{\left(1+\lambda^2\right)\left(1+\mu^2\right)}\\
   =&\kh{\frac{1}{2}-\frac{2\lambda\mu}{3(1+\lambda^2)(1+\mu^2)}}
   \kh{\bar{R}_{1212}+\bar{R}_{3434}} +\kh{\frac{2+\mu^2+\lambda^2}{2(1+\lambda^2)(1+\mu^2)}
   -\frac{3+\lambda\mu}{3(1+\lambda^2)(1+\mu^2)}}\bar{R}_{1313}\\
   &+\kh{\frac{1+2\lambda^2+\lambda^2\mu^2}{2(1+\lambda^2)(1+\mu^2)}
   -\frac{3\lambda^2+\lambda\mu}{3(1+\lambda^2)(1+\mu^2)}}\bar{R}_{1414}\\
   &+\kh{\frac{1+2\mu^2+\lambda^2\mu^2}{2(1+\lambda^2)(1+\mu^2)}
   -\frac{3\mu^2+\lambda\mu}{3(1+\lambda^2)(1+\mu^2)}}\bar{R}_{2323}\\
   &+\kh{\frac{\lambda^2+\mu^2+2\lambda^2\mu^2}{2(1+\lambda^2)(1+\mu^2)}
   -\frac{3\lambda^2\mu^2+\lambda\mu}{3(1+\lambda^2)(1+\mu^2)}}\bar{R}_{2424}
   +(n-4)\bar{K}_{max}\\
   & +\dfrac{\lambda\mu\left(\bar K(e_1+e_3, e_2+e_4) +\bar K(e_1-e_3, e_2-e_4)
+ \bar K(e_2+e_3, e_1-e_4) +\bar K(e_2-e_3, e_1+e_4)\right)}{6(1+\lambda^2)(1+\mu^2)}\\
\le
& \frac{1}{6(1+\lambda^2)(1+\mu^2)}\left[\kh{3(1+\lambda^2)(1+\mu^2)-4\lambda\mu}\cdot 2+\left(3(\lambda^2+\mu^2)-2\lambda\mu\right) +\left(3(1+\lambda^2\mu^2)-2\lambda\mu\right)\right.\\
&\left.\qquad\qquad\qquad\quad+\left(3(1+\lambda^2\mu^2)-2\lambda\mu\right)
+\left(3(\lambda^2+\mu^2)-2\lambda\mu\right) +16\lambda\mu\right]\bar{K}_{max}+(n-4)\bar{K}_{max}\\
=&(n-2)\bar{K}_{max},
\end{align*}
where in the last inequality, we have used the fact that all the coefficients are non-negative
for $\lambda,\mu\in[0,1]$, thus
we can replace $\bar R_{ijij}$ with $\bar K_{max}$.
Therefore,
\begin{align*}
&\frac{\varepsilon}{2}\sum_{i=3}^n\kh{\bar{R}(\varepsilon_1,
\varepsilon_i,\bar\varepsilon_1,\bar\varepsilon_i)
+\bar{R}(\varepsilon_2,\varepsilon_i,\bar\varepsilon_2,
\bar\varepsilon_i)}
+\frac{1-\varepsilon}{n-2}\sum_{i=3}^n\sum_{j=1}^n\bar{R}(\varepsilon_i,\varepsilon_j,\bar\varepsilon_i,\bar\varepsilon_j) \\
\leq&\varepsilon(n-2)\bar K_{max}+\frac{1-\varepsilon}{n-2}\cdot(n-2)(n-1)\bar K_{max}\\
=&(n-1-\varepsilon)\bar K_{max}.
\end{align*}
We complete the proof.
\end{proof}

If we take $\varepsilon =1$ in \autoref{epsilon}, we have \autoref{subric1}.
If we take $\varepsilon =\frac{2}{n-2}$ in \autoref{epsilon}, we have \autoref{subric2}.

We also can take $\varepsilon =\frac{2(n^2-6n+10)}{(n-4)(n^2-4n+2)}$ in \autoref{epsilon} to make the
coefficient $\delta(\varepsilon, n)$ to be minimal when $n\ge 6$.
\begin{cor}\label{cor4}
Suppose $M^n (n\ge 6)$ is a  closed simply connected submanifold of $\bar{M}^N$ satisfying
\begin{align*}
&\frac{Ric^{[2]}}{2}\ge \kh{n-1-\frac{2(n^2-6n+10)}{(n-4)(n^2-4n+2)}}\bar K_{max}+\frac{(n-2)(n-3)(n-4)n^2}{(n^2-4n+2)^2}\abs{H}^2,
\end{align*}
with strict inequality at some point,
then $M$ is   diffeomorphic to $\Lg{S}^n$.

\end{cor}

\begin{rem}
  It is easy to check, for $n\ge 6$,
  $$
  n-1-\frac{2(n^2-6n+10)}{(n-4)(n^2-4n+2)}< \frac{(n-2)(n-3)(n-4)n^2}{(n^2-4n+2)^2}< \frac{n(n-3)}{n-2}.
  $$
  Therefore, when $n\ge6$, \autoref{cor4} implies \autoref{subric2}.
\end{rem}

\end{document}